\documentclass{siamltex}
\usepackage{amssymb}
\usepackage{amsmath}
\usepackage{amscd}
\usepackage{latexsym}
\usepackage{longtable}
\usepackage{fancybox}
\usepackage{graphicx} 
\usepackage{ifthen}
\usepackage{stmaryrd}
\usepackage{color}
\usepackage{tikz}
\usepackage{harvard}
\usepackage{times}

\title{Optimal frequency sweep method in multi-rate circuit simulation}

\newcommand*\samethanks[1][\value{footnote}]{\footnotemark[#1]}
\author{Kai Bittner\thanks{University of Applied Sciences Upper Austria, Softwarepark 11, 4232 Hagenberg, Austria }\and
Hans Georg Brachtendorf\samethanks[1]}

\usepackage{amsfonts}
\usepackage{bm}

\newcommand{\I}{\mathcal{I}}

\newcommand{\RR}{\mathbb{R}}

\newcommand{\changed}[1]{{#1}}

\begin{document}

\maketitle

\begin{abstract}\par
  \textbf{Purpose -- RF circuits often possess a multi-rate behavior. Slow changing
          baseband signals and fast oscillating carrier signals often occur in the same circuit.
          Frequency modulated signals pose a particular challenge.}
  \par
  
  \textbf{Design/methodology/approach -- The ordinary circuit differential equations (ODEs) are first 
  rewritten by a system of
  (multi-rate) partial differential equations (MPDEs) in order to decouple the different time scales. 
   For an efficient simulation we need an optimal choice of a frequency dependent parameter.
   This is achieved by an additional smoothness condition. 
  }
  \par

  \textbf{Finding -- By incorporating the smoothness condition into the discretization, we obtain
          a nonlinear system of equations complemented by a minimization constraint. This problem
          is solved by a modified Newton method, which needs only little extra computational
          effort. The method is tested on a Phase Locked Loop with a frequency modulated input
          signal.}
  \par
  
  \textbf{Originality/value -- A new optimal frequency sweep method was introduced, which will permit
          a very efficient simulation of multi-rate circuits.}
\end{abstract}


\textbf{Paper type:} Research paper

\section{Introduction}

Widely separated time-scales occur in many
radio-frequency (RF) circuits such as mixers, oscillators, PLLs, etc., 
making the analysis with standard
numerical methods difficult and costly. Low frequency or baseband signals and high
frequency carrier signals often occur in the same circuit,
enforcing very small time-steps over a long time-period
for the computation of the numerical solution. The occurrence of widely separated
time-scales is also referred to as a multi-scale or multi-rate problem,
where classical
numerical techniques require prohibitively long run-times. 

A method to circumvent
this bottleneck is to reformulate the ordinary circuit \changed{equations, which are differentieal algebraic equations (DAEs),} 
as a system of partial DAEs (multi-rate PDAE).
The method was first presented in \cite{Bra94,BWL+96}, specialized to compute steady states.The technique was adapted
to the transient simulation of
driven circuits with a-priori known frequencies in \cite{NL96,Roy97}. 
A generalization
for circuits with a-priori unknown or 
time-varying frequencies was developed in
\cite{Bra97,BL98a,Bra2001}. This approach opened the door
to multi-rate techniques with frequency modulated signal sources
or autonomous circuits such as oscillators with a priori unknown fundamental
frequency \cite{BL98a,BL98,Bra2001,Hou04,Pulch08a,Pulch08b}.

Here, we present a new approach for the computation of a not a-priori known, time-varying 
frequency, which
is driven by the requirement to have a smooth multi-rate solution, crucial for the efficiency 
of the computation. 

\section{The multi-rate circuit simulation problem\label{MULTI_RATE}}

We consider circuit equations in the charge/flux oriented
modified nodal analysis (MNA) formulation, which yields a mathematical model
in the form of a system of differential-algebraic equations
(DAEs):
\begin{equation}
  \label{eq_MNA_charge}
  \tfrac{d}{dt}q\big(x(t)\big)
        + g\big(x(t)\big) = s(t).
\end{equation}
Here $x(t)\in\RR^n$ is the vector of node potentials and specific branch
voltages and $q(x)\in\RR^n$ is the vector of charges and fluxes. The vector
$g(x)\in\RR^n$ comprises static contributions, while $s(t)\in\RR^n$ contains the
contributions of independent sources.

To separate different time scales the problem is
reformulated as a multi-rate PDAE (cf. \cite{Bra2001,Hou04,Pulch08b,Pulch08a}), i.e.,
\begin{equation}\label{multirate}
\tfrac{\partial}{\partial \tau} q\big(\hat{x}(\tau,t)\big)
+\omega(\tau)\,\tfrac{\partial}{\partial t} q\big(\hat{x}(\tau,t)\big)
+g\big(\hat{x}(\tau,t)\big)=\hat{s}\big(\tau,t\big).
\end{equation}
If the new source term is chosen, such that
\begin{equation}\label{char_source}
s_\theta(t) = \hat{s}\big(t,\Omega_\theta(t)\big),
\end{equation}
where $\Omega_\theta(t)=\theta+\int_0^t \omega(s)\,ds$,
then a solution $\hat{x}$ of (\ref{multirate}) determines a family
$\{x_\theta:~\theta\in\RR\}$ of solutions for
\begin{equation}
  \label{singlerate}
  \tfrac{d}{dt}q\big(x(t)\big)
        + g\big(x(t)\big) = s_\theta(t),
\end{equation}
by
\begin{equation}\label{character}
x_\theta(t) = \hat{x}\big(t,\Omega_\theta(t)\big).
\end{equation}

Although the formulation (\ref{multirate}) is valid for any circuit,
it offers a more efficient solution only for certain types of
problems. This is the case if $\hat{x}(\tau,t)$ is periodic in $t$
and smooth with respect to $\tau$. Then, a semi-discretization with
respect to $\tau$ can be done resulting in a relative small number of
 periodic boundary problems
in $t$ for only a few discretization points $\tau_\ell$.
In the sequel we
will consider (\ref{multirate}) with periodicity conditions in
$t$, i.e., $\hat{x}(\tau,t)=\hat{x}(\tau,t+P)$ and suitable
initial conditions $\hat{x}(0,t)=X_0(t)$.
Here $P$ is an arbitrary but fixed period length, which results in a scaling of 
$\hat{x}(\tau,t)$ and $\omega(\tau)$
as we will see in the next section.

Note that for a given initial value for the original circuit equations
(\ref{eq_MNA_charge}) the
multi-rate formulation (\ref{multirate}) is not unique. First it does not
correspond to a single problem, but to a family of problems as
described in (\ref{singlerate}). This usually permits some freedom
in the choice of the multi-rate source term $\hat{s}(\tau,t)$ and for the initial
conditions. Furthermore, $\omega(\tau)$ can, in principle,
be chosen arbitrarily, but influences on the other hand, which choice 
of $\hat{s}(\tau,t)$ satisfies (\ref{char_source}).

In the sequel, we want to study how the above freedom can be used to
facilitate an efficient numerical solution of (\ref{multirate}). The
smoothness of $\hat{x}(\tau,t)$ is essential for the efficiency of
classical solvers. That is, the freedom we have for the formulation
of (\ref{multirate}) should be used to make the solution as smooth
as possible.

\section{Meaning and suitable choice of $\omega(\tau)$\label{OMEGA}}

It can be easily verified that $\omega(\tau)$ can be any positive function, if the source term
$\hat{s}(\tau,t)$ is chosen to satisfy (\ref{char_source}). 
Therefore, we will
first investigate, what effect different selections of
$\omega(\tau)$ have. Let $\omega_1(\tau)$ and $\omega_2(\tau)$ be two 
choices. If $\hat{x}_1(\tau,t)$ and $\hat{x}_2(\tau,t)$ both satisfy
(\ref{character}), then it is obvious that
\begin{equation}\label{omega2}
\hat{x}_2(\tau,t)=\hat{x}_1(\tau,t+S(\tau)),
\end{equation}
where $S(\tau) =\int_0^\tau \omega_1(s)-\omega_2(s)\,ds$. That is, for a
fixed $\tau$ changing $\omega(t)$ results in a phase shift of
$\hat{x}(\tau,t)$. A corresponding phase shift for the source term
has to be performed in order to satisfy (\ref{char_source}). 
The following Lemma describes how $\omega(\tau)$ affects the smoothness of 
$\hat{x}(\tau,t)$.

\begin{lemma}\label{smoothmr}
Let $L>\frac{P}{\omega(\tau)}$ for any $\tau>0$. If
\begin{equation}\label{smoothxhat}
\big\|\hat{x}(\tau+\delta,t)-\hat{x}(\tau,t)\big\|\le\varepsilon,
\end{equation}
 for $\delta\in(0,L]$, then 
\begin{equation}\label{smoothx}
\big\|x_\theta\big(\tau+T(\tau)\big)-x_\theta(\tau)\big\|\le\varepsilon,
\end{equation}
where $T(\tau)$ is the unique solution of 
$
P = \int_\tau^{\tau+T(\tau)} \omega(s)\,ds.
$
\end{lemma} 

\begin{proof}
Following (\ref{character}) we have 
\begin{eqnarray*}
x_\theta(\tau+T(\tau))&=& \hat{\changed{x}}\bigg(\tau+T(\tau), \theta+\int_0^{\tau+T(\tau)}\omega(s)\,ds\bigg)\\
&=& \hat{x}\bigg(\tau+T(\tau), \theta+\int_0^\tau\omega(s)\,ds+P\bigg)\\
&=& \hat{x}\bigg(\tau+T(\tau), \theta+\int_0^\tau\omega(s)\,ds\bigg).
\end{eqnarray*}
\changed{B}y the mean value theorem for integration we obtain
$$
P = T(\tau)\, \omega(\xi),\qquad \xi\in\big(\tau,\tau+T(\tau)\big).
$$
Thus, $T(\tau)<L$ and (\ref{smoothx}) follows from (\ref{smoothxhat}) by (\ref{character}).
\end{proof}

Lemma~\ref{smoothmr} states that we can only expect smoothness of $\hat{x}(\tau,t)$
in $\tau$ if $x_\theta$ is nearly $T(\tau)$-periodic in a neighborhood
of $\tau$. 

Since typical multi-rate signals behave locally like
a periodic signal, i.e.,
$x_\theta(t) \approx x_\theta(t+T(t^*)\changed{)}$ as long as $t$ is close to some $t^*$. 
That is, there should be a choice of $\omega(\tau)$ such that 
$\hat{x}(\tau_1,t)$ and $\hat{x}(\tau_2,t)$ do not differ much for sufficiently small
$\tau_2-\tau_1$.
This leads to the additional condition
\begin{equation}\label{min_cond}
\int_0^P \Big|\tfrac{\partial}{\partial\tau}\hat{x}(\tau,t)\Big|^2\,dt\
\to \min.
\end{equation}
in order to determine $\omega(\tau)$.
\changed{A similar condition is used in \cite{Pulch08a}). The difference to our approach is in the discretization
of the problem. While Pulch derives a phase condition from the minimization condition, we will discretize
the minimization condition directly, which will lead to a smaller linear system to be solved in the algorithm.
Furthermore, we are using a Rothe method for the discretization of the PDAE, which allows a completely adaptive solution.

A different approach is suggested in  \cite{Hou04}, where $\hat{x}(\tau,t)$
is replaced by $q\big(\hat{x}(\tau,t)\big)$ in (\ref{min_cond}). Although this seems the proper condition for an 
efficient discretization of  the derivative in $\tau$ in the multi-rate PDAE (\ref{multirate}), another problem occurs here.
Using an implicit multistep method for the discretization in $\tau$, we need a predictor
(as initial guess for Newton's method), which is usually based on the solution
of the previous time step. Here, a condition on  $\hat{x}(\tau,t)$ as in (\ref{min_cond}) occurs naturally. A condition
on $q\big(\hat{x}(\tau,t)\big)$ covers only contributions of the signal which are already smoothed 
(by capacitances or inductances), and might not provide a sufficient initial guess.} 

It turns out that condition (\ref{min_cond}) yields the expected result for amplitude and 
frequency modulation of a periodic signal\changed{, i.e., for typical multirate RF signals.}

\begin{lemma}\label{mr_modulated}
Assume 
$x_\theta(t)=a(t)\,y\big(\theta+\tilde{\Omega}(t)\big)$ with $y(t)=y(t+P)$ is a solution of 
\emph{(\ref{singlerate})}
with non-trivial $a(t)$ and $y'(t)$.
The solution of the corresponding multi-rate problem \emph{(\ref{multirate})} 
satisfies \emph{(\ref{min_cond})} if and only if
$\omega(\tau)=\tilde{\Omega}'(\tau)$
and $\hat{x}(\tau,t) = a(\tau)\,y(t)$.
\end{lemma}

\begin{proof}
According to (\ref{character}) and (\ref{omega2}) we can write the multi-rate solution as 
$\hat{x}(\tau,t) = a(\tau)\,y\big(t+S(\tau)\big)$, where $S(\tau)$ vanishes if and only if 
$\omega(\tau)=\tilde{\Omega}'(\tau)$. Taking the derivative with respect to $\tau$ we obtain
$$
\tfrac{\partial}{\partial\tau}\hat{x}(\tau,t)
=a'(\tau)\,y\big(t+S(\tau)\big)+a(\tau)\,y'\big(t+S(\tau)\big)\,S'(\tau).
$$ 
Note, that due to periodicity of $y(t)$ we have
$$
\int_0^P y\big(t+S(\tau)\big)\,y'\big(t+S(\tau)\big)\,dt 
= \tfrac{1}{2}\int_0^P \tfrac{d}{dt}y^2(t)\,dt
= \tfrac{1}{2}\left(y^2(P)-y^2(0)\right) = 0.
$$
Thus, the expression in (\ref{min_cond}), which has to be minimized becomes 
$$
 \big(a'(\tau)\big)^2\,\int_0^P y^2(t)\,dt
+\big(a(\tau)\,S'(\tau)\big)^2 \,\int_0^P \big(y'(t)\big)^2\,dt.
$$
Obviously the expression becomes minimal if $S'(\tau) = 0$, $\tau\in(\tau_1,\tau_2]$.
From (\ref{omega2}) we see that $S(0)=0$ and the statement is proven.
\end{proof}

Applying the approach to more general problems, permits to consider $\frac{\omega(t)}{P}$
as a generalization of the instantaneous frequency of the multi-rate problem, if (\ref{min_cond})
is satisfied. 

For the numerical solution of the multi-rate problem we have to discretize the MPDAE's (\ref{multirate})
together with the smoothness condition (\ref{min_cond}). The following section presents a new approach for the numerical solution of the multi-rate problem with unknown frequency parameter.

\section{Discretization\label{ROTHE}}

We discretize (\ref{multirate}) with respect to $\tau$ by a Rothe method using a linear multi step
method and obtain
\begin{equation}\label{rothe}
\sum_{i=0}^s \Bigg\{\alpha^k_i\, q\big(X_{k-i}(t)\big) + \beta^k_i
  \bigg(\omega_{k-i}\,\frac{d}{dt}q\big(X_{k-i}(t)\big)   
  + g\big(X_{k-i}(t)\big)-\hat{s}\big(\tau_{k-i},t\big)\bigg)\Bigg\}=0\quad
\end{equation}
complemented by the periodicity condition $X_k(t) = X_k(t+P)$.
Here we have to determine an approximation $X_k(t)$ of
$\hat{x}(\tau_k,t)$ from known, approximative solutions
$X_{k-i}(t)$ at previous time steps $\tau_{k-i}$. $i=1,\ldots,s$.
Examples are the trapezoidal rule with $s=1$, $\alpha^k_i = 1$ and
$\beta^k_i=\frac{2}{\tau_k-\tau_{k-1}}$ or GEAR-BDF of order $s$
with $\beta^k_i=\delta_{i,0}$ and suitable $\alpha^k_i$. That is,
$X_k$ is the solution of a periodic boundary value problem for the
DAEs given in (\ref{rothe}). 

Assuming that the $\tau_{k-i}$, $\omega_{k-i}$ and $X_{k-i}(t)$ are known and fixed
for $i>0$ we define
\begin{align*}
f_k&(x,t) :=\alpha^k_0 q(x) + \beta^k_0\Big(g(x)-\hat{s}\big(\tau_k,t\big)\Big)\\[-.5ex]
&+~\sum_{i=1}^s\Bigg\{ \alpha^k_i q\big(X_{k-i}(t)\big)
+ \beta^k_i \left(\omega_{k-i}\,\frac{d}{dt}q\big(X_{k-i}(t)\big)
                          +g\big(X_{k-i}(t)\big)-\hat{s}\big(\tau_{k-i},t\big)\right)\Bigg\},
\end{align*}
as well as $q_k(x) := \beta^k_0\,q(x)$.
Then (\ref{rothe}) becomes the periodic boundary value problem
\begin{equation}\label{envelope}
\omega\tfrac{d}{dt}q_k\big(x(t)) + f_k(x(t),t) = 0,\qquad x(t) = x(t+P).
\end{equation}

Furthermore, $\omega_k$ is an approximation 
for $\omega(\tau_k)$, which has to be determined during the computation (if not known in advance).
Since the midpoint rule states for a suitable choice of $\tau\in(\tau_1,\tau_2)$ that
$$
\tfrac{\partial}{\partial\tau}\hat{x}(\tau,t)=\hat{x}(\tau_2,t)-\hat{x}(\tau_1,t)
$$
we replace condition (\ref{min_cond}) by
\begin{equation}
\label{min_cond2}
\int_0^P\big|X_k(t)-X_{k-1}(t)\big|^2\,dt\to\min,
\end{equation}
i.e., we aim to minimize the
change of $X_k(t)$ from one time step to the next, which corresponds well to the original goal.

If we have a smooth envelope, we may be able to choose the step sizes $\tau_k-\tau_{k-1}$ much larger
than the period $T(\tau)=\frac{P}{\omega(\tau)}$. That means the number of time steps will become much smaller compared to the number
of oscillations of the carrier signal. This can lead to an essential gain in computation time compared
to transient analysis, since the computation cost for one period is comparable to one envelope time step. 

The periodic problem (\ref{envelope})
can be solved by a collocation or Galerkin method, where $X_k(t)$ is expanded
in a periodic basis $\{\phi_k\}$ 
(as a Fourier, B-spline, or wavelet basis) and tested at collocation points or integrated
against test functions (see e.g. \cite{BiDau10a,BiDau10b}). 
This leads to a nonlinear system 
\begin{equation}\label{nonlinear}
F(\bm{c},\omega) = 0
\end{equation}
of equations for the coefficients
$\bm{c}=\bm{c}^{(k)}=(c^{(k)}_\ell)_\ell$ of the basis expansion 
$X_k(t)=\sum_\ell c^{(k)}_\ell\,\phi_\ell(t)$ and the unknown frequency parameter $\omega$. 
Here, condition
(\ref{min_cond2}) is replaced by a condition on the expansion coefficients, namely
\begin{equation}\label{min_cond3}
\sum_\ell\|c^{(k)}_\ell-c^{(k-1)}_\ell\big\|^2_2\to \min.
\end{equation} 
This means, that we are still minimizing the distance of $X_k(t)$ and $X_{k-1}(t)$, although with
respect to a slightly modified measure.

The nonlinear system (\ref{nonlinear}) is solved by a modified Newton's method. 
The linearization of the problem yields the under-determined system
\begin{equation}\label{newton_ud}
\bm{A\,d}_c+d_\omega\,\bm{z} =\bm{b},
\end{equation}
with the right hand site 
$\bm{b}=F(\bm{c},\omega)$, the Jacobian $\bm{A}= \changed{\frac{\partial}{\partial \bm{c}}} F(\bm{c},\omega)$, and 
$\bm{z} = \changed{\frac{\partial}{\partial \omega}} F(\bm{c},\omega)$. Furthermore, $\bm{d}_c$ and $d_\omega$ are the updates
in the Newton iteration
$$
\bm{c}^{(k,j)}=\bm{c}^{(k,j-1)}-\bm{d}_c,\qquad \omega_{k,j}=\omega_{k,j-1}-d_\omega,
$$
\changed{with an initial guess chosen as $\bm{c}^{(k,0)}=\bm{c}^{(k-1)}$ and $\omega_{k,0}=\omega_{k-1}$.}

To find a solution satisfying (\ref{min_cond3}) we rewrite (\ref{newton_ud}) as
$$
\bm{d}_c=\tilde{\bm{b}} - d_\omega\,\tilde{\bm{z}},
$$
where $\tilde{\bm{b}}=\bm{A}^{-1}\,\bm{b}$ and $\tilde{\bm{z}}=\bm{A}^{-1}\,\bm{z}$ are computed 
by solving the corresponding linear systems. This needs a LU-decomposition of $\bm{A}$ and two 
forward-backward substitution, i.e., the computational costs are only slightly higher as for 
$\omega(\tau)$ fixed in advance. Possible higher step sizes and faster convergence of Newton's method
justify this extra cost if a good choice for $\omega(\tau)$ is not available in advance.
In the $j$-th Newton step $\bm{c}^{(k,j)} = \bm{c}^{(k,j-1)}-\bm{d}_c$ we determine the Newton 
correction, which satisfies (\ref{min_cond3})
as follows. We want
$$
\|\bm{c}^{(k,j)}-\bm{c}^{(k-1)}\big\|^2_2
=\|\underbrace{\bm{c}^{(k,j-1)}-\bm{c}^{(k-1)}-\tilde{\bm{b}}}_{\bm{a}} 
   + d_\omega\,\tilde{\bm{z}}\|^2_2
$$
to be minimal.
Obviously the minimum is attained for
$
d_\omega = -\frac{\tilde{\bm{z}}^T\bm{a}}{\tilde{\bm{z}}^T\tilde{\bm{z}}}.
$

\section{Numerical test --- A phase locked loop\label{NUM_TEST}}

\changed{The method has been tested on several circuits. 
We show results from the
multi-rate simulation of a Phase Locked Loop (PLL) with a frequency modulated input signal.
The phase of the output signal is synchronized  with the phase of the input ``reference'' signal.}
This is achieved, by comparing the phases of the ``reference'' signal and the
output or ``feedback'' signal in a phase detector. The result is then filtered,
in order to stabilize the behavior of the PLL, and fed to a Voltage Controlled Oscillator (VCO)
whose frequency depends on the (filtered) phase difference. The output is fed back
(possibly through a frequency divider) to the phase detector. If the PLL is locked, the phases
of reference and feedback signal are synchronized. Thus both signal will have (almost) 
the same local frequency.

\begin{figure}[h]
  \centering
  \includegraphics[width=0.49\textwidth]{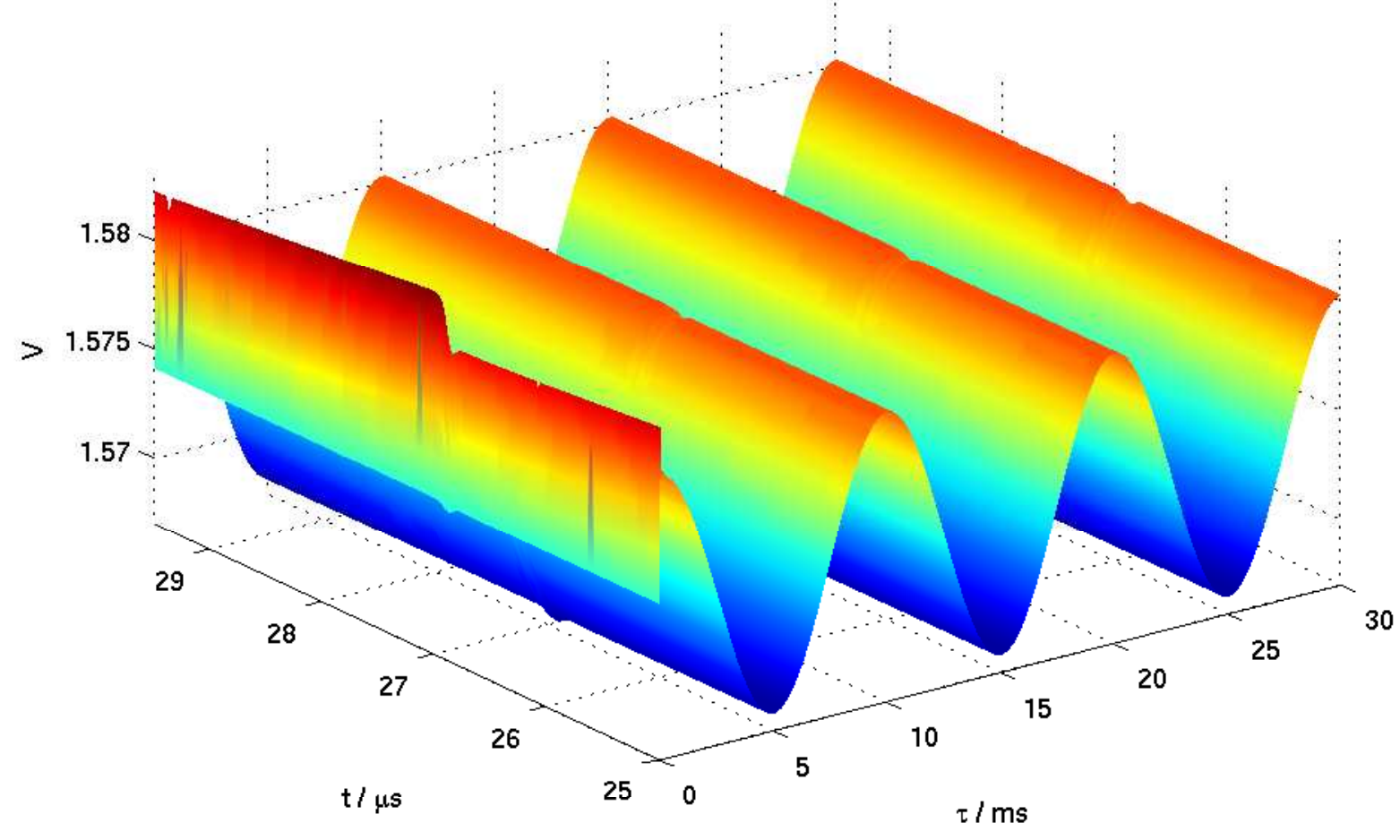}\hfill\includegraphics[width=0.49\textwidth]{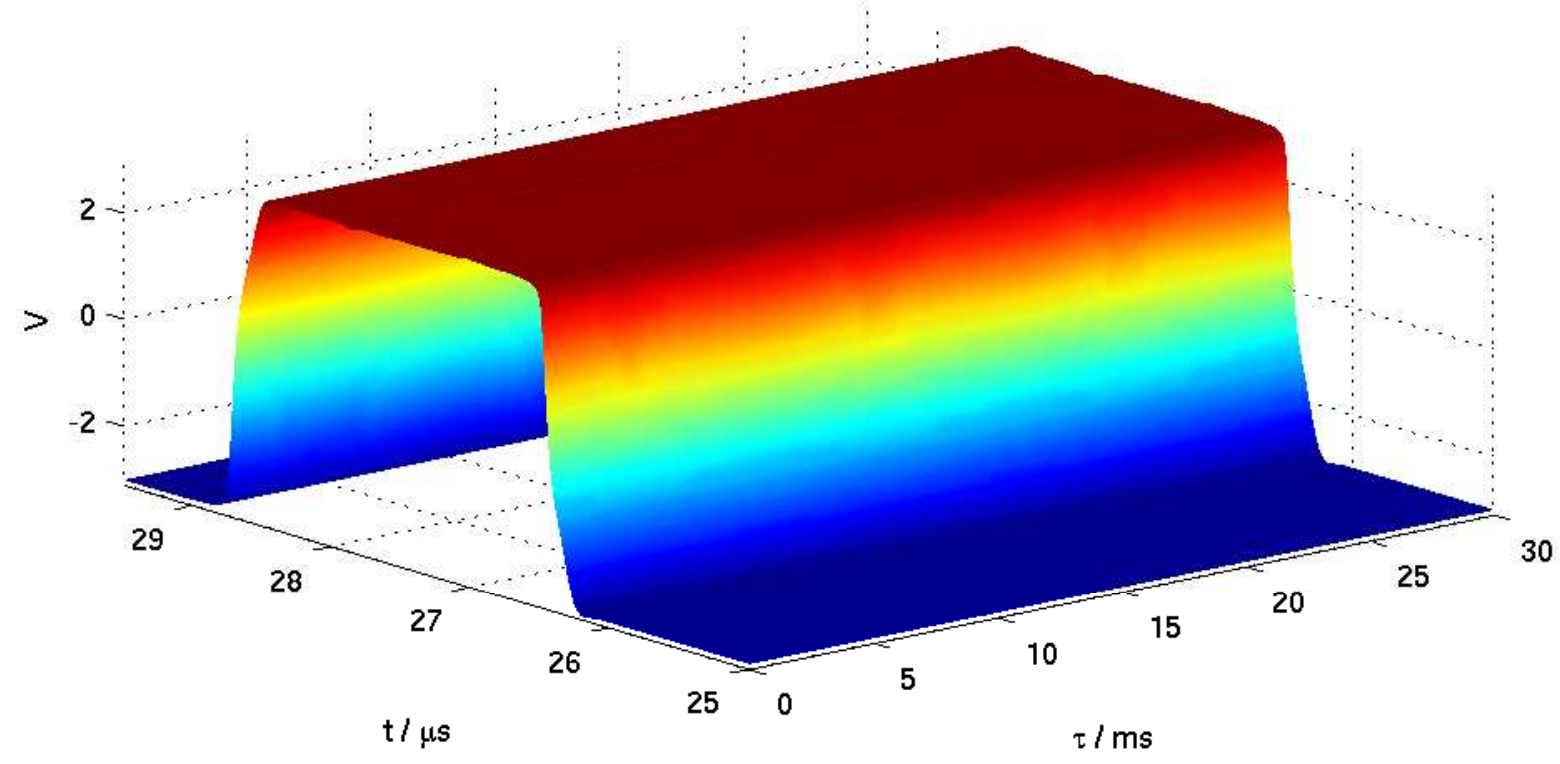}
  \caption{Control and feedback signal of the PLL}
  \label{pllcontrol}
\end{figure}

We have tested a PLL (containing 145 MOSFETs and 80 unknowns)
with a frequency modulated sinusoidal signal with frequency 222kHz.
The baseband signal is also sinusoidal with frequency 200Hz and frequency deviation 200Hz.
\changed{We have chosen this example, since it was particular challenging, due to its size, but also
due its digital-like behavior, which requires an efficient adaptive representation of the solution, as well as
a very accurate estimate of the frequency parameter $\omega(\tau)$. Although the baseband signal is 
known to us after locking of the PLL, this information is not provided to the algorithm.
This permits a verification of the  $\omega(\tau)$ determined by our algorithm. Before the locking
of the PLL the optimal $\omega(\tau)$ is not known to us, and the estimate of the algorithm
proved essential for this part of the simulation to be successful.}

\begin{figure}[h]
  \centering
  \includegraphics[width=0.6\textwidth]{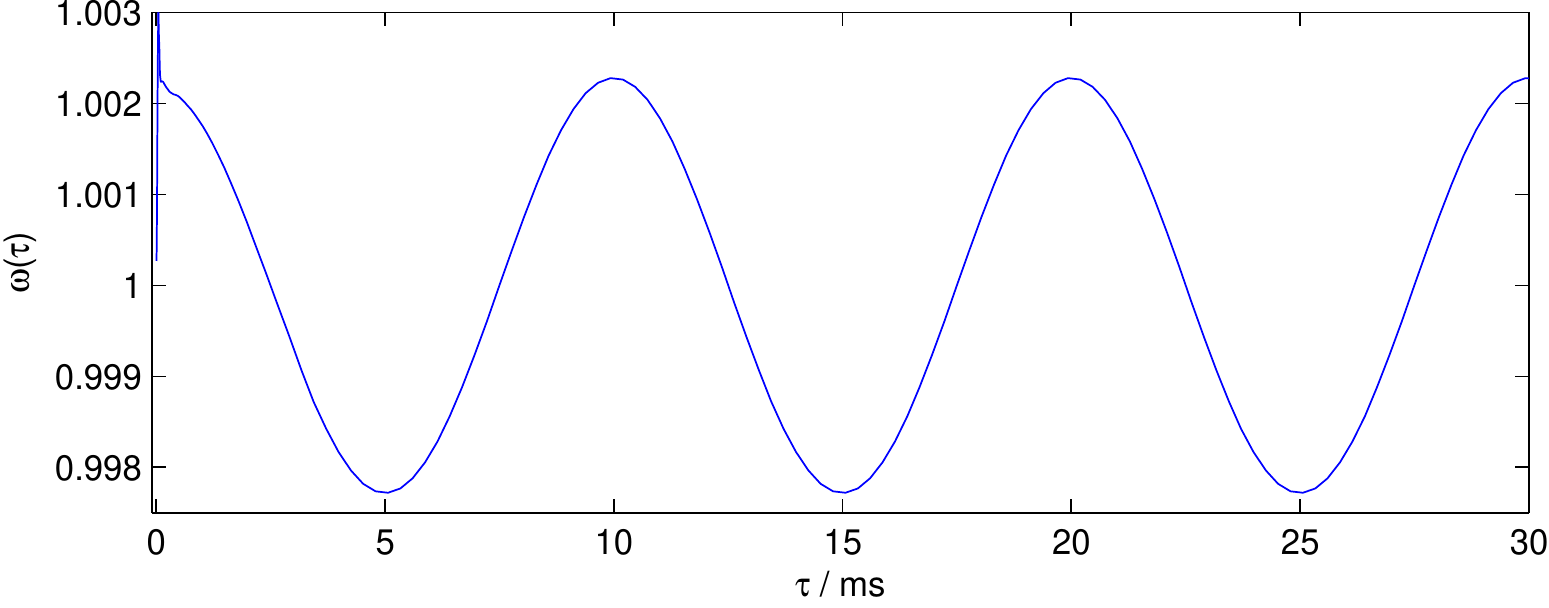}
  \caption{Frequency parameter $\omega(\tau)$}
  \label{pllomega}
\end{figure}

Due to the chosen circuit design most of the signals show a digital behavior such that the solution
can not efficiently be expanded by a Fourier basis. Thus, using the Harmonic Balance method is not suited
for the solution of the problem, in particular since the use of a frequency divider results in a wide 
frequency range of the solution.
Therefore, the periodic problem is solved by an adaptive 
spline wavelet method described in \cite{BiDau10a,BiDau10b}. The quality of the estimate of
$\omega(\tau)$ is essential for the efficiency of the method, because this results in a good initial guess
for the solution and the spline grid in Newton's method.  

In Figure~\ref{pllcontrol} (left) one can see the control signal for the VCO (the filtered output
of the phase detector), which controls the frequency of the oscillator. Note, that after the locking phase at the beginning the envelope
corresponds to the baseband signal, while the carrier signal is due to the filtering almost
constant.
Figure~\ref{pllcontrol} (right) shows the digital feedback signal. Note that $\omega(\tau)$ was chosen 
to satisfy the condition (\ref{min_cond3}). The plot of $\omega(\tau)$ in Figure~\ref{pllomega}
fits almost perfectly the baseband signal, \changed{after the PLL is locked}. 
In the light of Lemma~\ref{mr_modulated}, the graphs in
Figures~\ref{pllcontrol} and \ref{pllomega} show 
indeed a frequency modulated box shaped signal.
The optimal choice of $\omega(\tau)$ results in high smoothness with respect to $\tau$.
This allows to do the shown envelope simulation using only 74 envelope time steps, while a 
corresponding transient analysis would contain 3100 oscillations. 
The wavelet envelope simulation of this circuit was
done in 7\,min. However, a comparable transient simulation on the same simulator needed 33\,hours. 

\section{Conclusion}

We have presented a new method for an accurate estimate of an unknown frequency parameter
in multi-rate envelope simulations, which permits an efficient simulation even for challenging problems.



\end{document}